\documentclass[12pt]{amsart}

\usepackage{amsmath}

\usepackage{bbm}
\usepackage{enumerate}
\usepackage{paralist}
\usepackage{amssymb}

\usepackage{stmaryrd}

\numberwithin{equation}{section}

\newtheorem{theorem}{Theorem}[section]

\newtheorem{lemma}[theorem]{Lemma}
\newtheorem{corollary}[theorem]{Corollary}

\theoremstyle{definition}

\newtheorem*{teoremNN}{Theorem}

\newcommand \NN{\mathbbm{N}}

\newcommand \Iso{{\rm Iso}}
\newcommand \EIso{{\rm EIso}}

\newcommand \dist{{\rm dist}}
\newcommand \diam{{\rm diam}}

\newcommand {\db}[1]{\llbracket #1 \rrbracket}

\newcommand \BB{\mathcal{B}}

\newcommand \EE{\mathcal{E}}
\newcommand \FF{\mathcal{F}}

\begin{document}

\title[Generic elements in isometry groups]{Generic elements in isometry groups of Polish ultrametric spaces}
\author{Maciej Malicki}

\address{Department of Mathematics and Mathematical Economics, Warsaw School of Economics, al. Niepodleglosci 162, 02--554 Warsaw, Poland}
\email{mamalicki@gmail.com}
\date{Dec 23, 2013}
\subjclass[2000]{(primary) 54H15, (secondary) 20E34, 20E22}

\begin{abstract}
This paper presents a study of generic elements in full isometry groups of Polish ultrametric spaces. We obtain a complete characterization of Polish ultrametric spaces $X$ whose isometry group $\Iso(X)$ has a neighborhood basis at the identity consisting of open subgroups with ample generics. It also gives a characterization of the existence of an open subgroup in $\Iso(X)$ with a comeager conjugacy class.

We also study the transfinite sequence defined by the projection of a Polish ultrametric space $X$ on the ultrametric space of orbits of $X$ under the action of $\Iso(X)$.
\end{abstract}

\maketitle
\section{Introduction}
In  this paper we study full isometry groups of Polish ultrametric spaces, equipped with the topology of pointwise convergence. We are mainly interested in the conjugation actions, and in the notion of ample generics.

A Polish (that is, separable and completely metrizable) topological group $G$ has ample generics if the diagonal action of $G$ on $G^{n+1}$ by conjugation has a comeager orbit for every $n \in \NN$. This notion was first studied by Hodges, Hodkinson, Lascar and Shelah in \cite{HoHo}, and later by Kechris and Rosendal in \cite{KeRo}. It is a very strong property: a group $G$ with ample generics, or even containing an open subgroup with ample generics, has the small index property, every homomorphism from $G$ into a separable group is continuous, every isometric action of $G$ on a separable metric space is continuous, and there is only one Polish group topology on $G$.

Surprisingly, many natural examples of Polish groups do have ample generics, and in recent years much effort has been put in detecting them. One of fundamental results in this area (\cite[Theorem 6.2]{KeRo}) provides a full characterization of Polish subgroups of the group $S_\infty$ of all permutations of the natural numbers which have ample generics. It is formulated in terms of Fra\"{i}ss\'e limits. Unfortunately, even if $G$ is a subgroup of $S_\infty$ (let alone it is not) this approach may not work because of a lack of a natural description of $G$ in terms of automorphisms of a countable structure, which is required in this framework. Nevertheless, both positive and negative results have been obtained despite this obstacle. For example, Kechris (unpublished but see \cite[Theorem 12.5]{GlWe}) proved that the isometry group $\Iso(\mathbbm{U})$ of the Polish Urysohn space $\mathbbm{U}$ does not have ample generics (in fact, every conjugacy class of $\Iso(\mathbbm{U})$ is meager), while Solecki \cite{So} showed that ${\rm Aut}(\mathbbm{U}_0)$, where $\mathbbm{U}_0$ is the rational Urysohn space regarded as a discrete relational structure, has ample generics.

Here, we are interested in an analogous investigation in the realm of Polish ultrametric spaces, that is, Polish metric spaces satisfying a strong version of the triangle inequality:
\[ d(x,z) \leq \max \{ d(x,y),d(y,z) \}.\]

We obtain a complete characterization of Polish ultrametric spaces $X$ whose isometry group $\Iso(X)$ has a neighbourhood basis at the identity consisting of open subgroups with ample generics (Theorem  \ref{th:main}), and we show that if $\Iso(X)$ does not such a basis, then it does not contain an open subgroup with a comeager conjugacy class (this, actually, also covers the case of countable ultrametric spaces regarded as discrete relational structures.) In particular, we show that isometry groups of all Polish ultrametric Urysohn spaces have ample generics (Corollary \ref{co:Ur}.) Proofs of these results are based on the fact that all Polish ultrametric spaces satisfy a weak version of the Hrushovski property (Lemma \ref{le:ext1}.)

As an application of our study, we provide an example of a Polish group with ample generics and countable cofinality (Corollary \ref{co:co}.) Recall that a group has countable cofinality if it is a union of a countably infinite, strictly increasing sequence of its subgroups. Theorem 6.12 in \cite{KeRo} says that a Polish group with ample generics is \emph{not} a union of any countable strictly increasing sequence of its \emph{non-open} subgroups, so ample generics almost preclude countable cofinality. It is natural to ask then whether these two notions are consistent with one another. We show that actually they are.

Another somewhat surprising upshot is a Polish ultrametric space which is extremely small in some sense but whose isometry group has ample generics (Corollary \ref{co:n-Ur}.) In this space $X$, all maximal polygons $P$ (that is, $P \subseteq X$ such that $d(x,y)=r$ for some fixed $r>0$ and every $x,y \in P$) have size $n=2$. For any other finite $n>1$, isometry groups of such spaces do not even contain an open subgroup with a comeager conjugacy class.

In \cite[Theorem 4.4]{AtGl}, the authors constructed a locally compact group with a dense conjugacy class. On the other hand, P.Wesolek \cite{We} has recently proved that locally compact Polish groups cannot have comeager conjugacy classes. We present an example of a locally compact group $G$ such that every diagonal action of $G$ on $G^{n+1}$, $n \in \NN$, by conjugation has a dense orbit (Corollary \ref{co:de}.)  Moreover, $G$ is given by a simple, explicit formula (the construction in \cite{AtGl} is inductive.)

In the last section we analyze the canonical projection of a Polish ultrametric space $X$ on the Polish ultrametric space of orbits of the natural action of $\Iso(X)$ on $X$, as described in \cite{MaSo}. Iterating this construction leads to a sequence of Polish ultrametric spaces, and their projections. We prove the existence of the limit step, which implies that every Polish ultrametric space can be canonically reduced to a rigid one (Theorem \ref{th:ri}.)

\section{Terminology and basic facts}

Let us start with a short review of basic facts and definitions exploited in the paper. A metric space $X$ is called \emph{Polish} if it is separable and complete. It is called \emph{ultrametric}, or \emph{non-archimedean}, if it satisfies a strong version of the triangle inequality:
\[ d(x,z) \leq \max \{ d(x,y),d(y,z) \}, \]
for all $x,y,z \in X$.

Typical examples of ultrametric spaces are
\begin{enumerate}[i)]
\item the Baire space, that is, the family $\NN^\NN$ of all sequences of the natural numbers with metric $d$ defined by
\[ d(x,y)=\max \{2^{-n}: x(n) \neq y(n) \}  {\rm \ for \ } x \neq y,\]
\item a valued field $K$ with valuation $| \cdot |:K \rightarrow \mathbbm{R}$ and metric
\[ d(x,y)=|x-y|,\]
in particular the field $\mathbbm{Q}_p$ of $p \, $-adic numbers is an ultrametric space,
\item the space of all rays in an $\mathbbm{R}$-tree $T$ starting from a fixed point $t \in T$ can be canonically given an ultrametric structure.
\end{enumerate}

By a ball of radius $r$ in an ultrametric space $X$ we will always mean a set of the form
\[ \{ x' \in X: d(x,x')<r\} \] 
for some $x \in X$. Even if we do not specify the radius, we assume that it has one. Note that two balls may be equal as sets even though they have different radii.
 
An easy to prove but fundamental property of every ultrametric space $X$ is that if $B, B'$ are balls in $X$, then
\[ B \subseteq B' {\rm \ or \ } B \subseteq B' {\rm \ or \ } B \cap B'=\emptyset.\]
In particular, every triple $x,y,z \in X$ realizes only two distances:

\begin{equation}
\label{eq0}
d(y,x)=d(x,z) {\rm \ or \ } d(x,y)=d(y,z) {\rm \ or \ } d(x,z)=d(z,y). 
\end{equation}
This observation will be repetitively used. It implies that  every element of a ball in $X$ is its center, and that the \emph{set of distances}
\[ \{ d(x,y): x,y \in X, \, x\neq y \}\]
of a separable ultrametric space $X$ is always countable.

%

We say that a metric space $X$ is \emph{ultrahomogeneous} if every isometric bijection between finite subsets of  $X$ can be extended to an isometry of  $X$. It is well known that every Polish ultrametric space $X$ with the set of distances contained in a countable $R \subseteq \mathbbm{R}^{>0}$ can be isometrically embedded in an ultrahomogeneous Polish ultrametric space $X_R$ defined as follows:
\[ X_R=\{ x \in \NN^R: \forall r>0 \, (\{ r' \in R: r'>r {\rm \ and \ } x(r') \neq 0 \} { \rm \ is \ finite} ) \}, \]
\[ d(x,y)=\max \{ r \in R: x(r) \neq y(r) \} {\rm \ for \ } x,y \in X_R, \ x \neq y.\]
The space $X_R$ is called a Polish ultrametric Urysohn space.

One can also consider Polish ultrametric $n$-Urysohn spaces $X^n_R \subseteq X_R$, where $n$ is a natural number, consisting of sequences in $\{0, \ldots, n-1\}^R$ (instead of $\NN^R$.) Clearly, all maximal $r$-\emph{polygons} in $X_R$, that is, sets $P \subseteq X_R$ such that $d(x,y)=r$ for every $x,y \in P$, are infinite. Every Polish ultrametric $n$-Urysohn space $X^n_R$ is ultrahomogeneous, and universal for Polish ultrametric spaces $X$ with the set of distances contained in $R$ and such that every maximal polygon in $X$ has size at most $n$.

For a Polish metric space $X$ the group $\Iso(X)$ of all isometries of $X$, that is, all distance preserving bijections $\phi:X \rightarrow X$, is a Polish topological group when equipped with the topology of pointwise convergence. This is the only topology on $\Iso(X)$ we consider in this paper. In the case that $X$ is ultrametric, sets of the form
\[ \tau(g,\mathcal{B})=\{ g' \in \Iso(X): g'[B]=g[B] \mbox{ for each } B \in \mathcal{B} \}, \]
where $g \in \Iso(X)$, $\mathcal{B}$ is a finite family of balls in $X$, form a basis of this topology. Actually, it suffices to consider finite families of pairwise disjoint balls, that is, \emph{FD families}. If $\BB$, $\BB'$ are FD families such that $\tau(Id, \BB') \subseteq \tau(Id, \BB)$, we say that $\BB'$ \emph{strengthens} $\BB$.

Let $A,B \subseteq X$, and let $f:A \rightarrow B$ be an isometric mapping. We say that $f$ is a \emph{partial isometry} of $X$. If $f$ can be extended to an isometry of $X$, we say that $f$ is $X$-\emph{extendable}. By $\EIso(A)$ we denote the group of all $X$-extendable partial isometries $f:A \rightarrow A$. Finally, we say that two sets $Y, Y' \subseteq X$ are \emph{similar} if there exists $g \in \Iso(X)$ such that $g[Y] = Y'$. 

If every extension of an $X$-extendable partial isometry $f$ to an isometry of $X$ induces a fixed permutation of an $FD$ family $\mathcal{B}$, we say that $f$ induces a permutation of $\BB$.

If $C$ is a ball in $X$, and $\BB$ is an FD family, we define
\[ \mathcal{O}(C)=\{ g[C]:g \in \Iso(X) \}, \, \mathcal{O}^\BB(C)=\{ g[C]:g \in \tau(Id,\BB) \}. \]
By (\ref{eq0}), both of these families consist of pairwise disjoint balls.

Let $B$ a ball of radius $r$, and let $x  \in X$. The collection of all balls $B'$ of radius $r$ which are similar to $B$, and such that $\dist(B,B') = r$, is denoted by $[B]$. The orbit of $x$ under the action of $\Iso(X)$ on $X$ is denoted by $\db{x}$.

Observe that for every  ball $B$ the group of those elements of $\Iso(X)$ that permute $[B]$, acts on $[B]$ as a full symmetric group.

Let $\BB$ be an FD family in $X$. If $\mathcal{O}^\BB(C)$ is infinite or a singleton for every ball $C$ in $X$, we say that $X$ has the $\BB$-\emph{extremality property}. 
Also, if  a ball $C$ in $X$ is such that $\mathcal{O}^\BB(C)$ is infinite or a singleton, then we say that $C$ is $\BB$-\emph{extremal in} $X$.  
Otherwise, $C$ is $\BB$-\emph{non-extremal in} $X$. If $C$ is a ball of radius $r$, which is $\BB$-non-extremal in $X$, and $C'$ is $\BB$-extremal in $X$ for every ball $C'$ of radius $r'>r$ containing $C$, then we will say that $C$ is \emph{radius-maximal} $\BB$-\emph{non-extremal in} $X$ or, shortly, \emph{r-maximal} $\BB$-\emph{non-extremal in} $X$.

If $\BB$ is empty, we simply say that $X$ has the extremality property, and $C$ is extremal in $X$.

%
%
%


Finally, a Polish group $G$ has \emph{ample generics} if each \emph{diagonal conjugation action} of $G$, that is, an action of $G$ on $G^{n+1}$, for some $n \in \NN$, defined by
\[ g.(g_0,\ldots, g_n)=(gg_0g^{-1}, \ldots, gg_ng^{-1}), \]
$g, g_0, \ldots, g_n \in G$, has a comeager orbit.
Main properties of groups with ample generics are summed up below (Theorems 6.9, 6.12, 6.24 and 6.25 in \cite{KeRo}):
\begin{teoremNN}
Let $G$ be a Polish group containing an open subgroup $H$ with ample generics. Then
\begin{enumerate}[1)]
\item every homomorphism from $G$ into a separable topological group is continuous,
\item there is only one Polish group topology on $G$,
\item $G$ has the small index property,
\item $H$ is not the union of countably many non-open subgroups.
\end{enumerate}
\end{teoremNN}

\section{Ample generics} 

The following lemma is well known. We prove it for the sake of completeness.

\begin{lemma}
\label{le:fo}
Let $X$ be a Polish space, and let $G$ be a Polish group continuously acting on $X$.
\begin{enumerate}
\item Suppose that $x \in X$ is such that for every open neighborhood of the identity $V \subseteq G$, the set $V.x$ is somewhere dense in $G.x$. Then the orbit of  $x$ is comeager in $\overline{G.x}$,
\item suppose that there exists a non-empty open $U \subseteq X$ such that for every $x \in U$ there exists an open neighborhood of the identity $V \subseteq G$ such that $V.x$ is meager in $X$. Then there is no comeager orbit in $X$.
\end{enumerate}
\end{lemma}

\begin{proof}
We prove Point (1). Suppose that $G.x$ is not comeager in $\overline{G.x}$. To derive a contradiction, we will find an open neighborhood of the identity $V \subseteq G$ such that $V.x$ is nowhere dense in $\overline{G.x}$. Since $G.x$ is analytic, and so it has the Baire property in $\overline{G.x}$, there exists an open set $U \subseteq X$ such that $G.x \cap U$ is non-empty and meager in $\overline{G.x}$. In other words, $U \cap G.x \subseteq \bigcup_n F_n$, where each $F_n$ is closed and nowhere dense in $\overline{G.x}$. The set
\[ W=\{ g \in G: g.x \in U \} \]
is open and non-empty in $G$, so there exists a non-meager $W' \subseteq W$ such that $W'.x \subseteq F_{n_0}$ for some $n_0$. By continuity of the action, $\overline{W'}.x \subseteq F_{n_0}$ as well. Since $\overline{W'}$ has non-empty interior, there exists $g \in G$ and an open neighborhood of the identity $V$ such that $gV.x \subseteq F_{n_0}$. Therefore $V.x \subseteq g^{-1}F_{n_0}$, and $V.x$ is nowhere dense in $\overline{G.x}$.

Concerning Point (2), every comeager orbit must intersect $U$, so it suffices to prove that no orbit of an element of  $U$ is comeager. Fix $x \in U$, and $V \subseteq G$ such that $V.x$ is meager. Clearly, $gV.x$ is meager for every $g \in G$ . Find $g_0, g_1, \ldots \in G$ such that $\bigcup_n g_nV=G$. Then $G.x=\bigcup_n g_n V.x$, so $G.x$ is meager.

\end{proof}

\section{Extending partial isometries in ultrametric spaces} 

The next lemma is a simplified version of \cite[Lemma 4.4]{MaSo}.

\begin{lemma}
\label{le:ext0}
Suppose that $X$ is a Polish ultrametric space, $Y,Y' \subseteq X$ are finite, and $f:Y \rightarrow Y'$ is an isometric bijection such that $\db{y}=\db{f(y)}$ for every $y \in Y$. Then $f$ is $X$-extendable.
\end{lemma}

\begin{proof}
Fix  $x \in  X \setminus Y$, and $y \in Y$ such that
\[ d(x,y)=\dist(x,Y)=r.\]
Because $\db{y}=\db{f(y)}$ for every $y \in Y$, $x$ witnesses that there exists $x' \in \db{x} \setminus Y'$ such that 
\[ d(x',f(y))=\dist(x',Y')=r.\]
We leave to the reader checking that $f \cup (x,x')$ is isometric (here, (\ref{eq0}) is crucial.)

Proceeding as above, we can extend $f$ to an isometry $g$ of a countable, dense subset of $X$. Then the unique extension of $g$ to an isometry of $X$ witnesses that $f$ is $X$-extendable.
\end{proof}

In \cite{So}, it is proved that for every finite metric space $Y$ there exists a finite metric space $Z$ such that $Y \subseteq Z$, and every partial isometry of $Y$ can be extended to an isometry of $Z$.  It follows that for every finite subset $Y$ of the Urysohn space $\mathbbm{U}$ there exists a finite $Z \subseteq \mathbbm{U}$ such that $Y \subseteq Z$, and every partial isometry of $Y$ can be extended to an isometry of $\mathbbm{U}$ that setwise fixes $Z$. Thus, $\mathbbm{U}$ has the \emph{Hrushovski property}. 
The Hrushovski property does not hold for arbitrary metric spaces, it turns out though that every ultrametric space shares a weak version of it.'

\begin{lemma}
\label{le:ext1}
Let $X$ be an ultrametric space. For every finite $Y \subseteq X$ there exists a finite $Z \subseteq X$ such that $Y \subseteq Z$, and every $X$-extendable partial isometry of $Y$ can be extended to an $X$-extendable isometry of $Z$. Moreover, if $\left| Y \right|>1$, then $Z$ can be chosen so that for every $z \in Z$ there is $y \in Y$ such that $d(z,y)<\diam(Y)$.
\end{lemma}

\begin{proof}
We prove the lemma by induction on the size of $Y$. Fix a finite $Y \subseteq X$.  If $Y$ is a singleton, then $Z=Y$ is as required. Suppose that $\left| Y \right|>1$, and that the lemma holds for sets of size $< \left| Y \right|$. Let $r=\diam(Y)$, and let $\EE$ be the family of all  balls of radius $r$, which have non-empty intersection with $Y$. Then $\left| \EE \right|>1$, and, because of (\ref{eq0}), $\dist(E,E')=r$ for every $E,E' \in \EE$ such that $E \neq E'$.

Put $Y_E= Y \cap E$ for $E \in \EE$. Define the relation $\sim_0$ on elements of $\EE$ as follows: $E \sim_0 E'$ if there exist $x \in Y_E$, $x' \in Y_{E'}$ such that $\db{x}=\db{x'}$. Let $\sim$ be the transitive closure of $\sim_0$. Then $\sim$ is an equivalence relation such that $E \sim E'$ implies that $E,E'$ are similar, while $E \not \sim E'$ implies that there is no $X$-extendable mapping between $Y_E$ and $Y_{E'}$.

Suppose that the relation $\sim$ has more than one equivalence class, and let $\mathcal{C}$ be the family of all equivalence classes of $\sim$. For $\FF \in \mathcal{C}$ put $F_{\FF}=\bigcup \FF$, $Y_{\FF}=Y \cap F_{\FF}$. Then $\left| Y_{\FF} \right|< \left| Y \right|$ for every $\FF \in \mathcal{C}$, so, by the inductive assumption, we can fix $Z_{\FF}$ as in the statement of the lemma, assuming additionally that $Z_\FF=Y_\FF$ if $Y_\FF$ is a singleton. Observe that the facts that each $F_{\FF}$ is a union of balls of radius $r$, and that for every $z \in Z_\mathcal{F}$ there is $y \in Y_\mathcal{F}$ such that $d(z,y)<\diam(Y_\mathcal{F})$, imply that $Z_{\FF} \subseteq F_{\FF}$. Therefore $\dist(Z_{\FF},Z_{\FF'})=r$ for any two distinct $\FF$, $\FF' \in \mathcal{C}$.  In this situation (\ref{eq0}) easily implies that for any collection of $g_\FF \in \EIso(Z_\FF)$, $\FF \in \mathcal{C}$, the mapping $\bigcup_{\FF \in \mathcal{C}} g_\FF$ is isometric, and thus it is also $X$-extendable. 

Fix an $X$-extendable partial isometry $f$ of $Y$. By the definition of the relation $\sim$, $f[Y_\FF] \subseteq Y_\FF$ for every $\FF \in \mathcal{C}$, that is, each $f \upharpoonright Y_\FF$ is an $X$-extendable partial isometry of $Y_\FF$. Therefore we can extend each $f \upharpoonright Y_\FF$ to some $g_\FF \in \EIso(Z_\FF)$. By the above remarks, $\bigcup_{\FF \in \mathcal{C}} g_\FF$ is $X$-extendable, and it it extends $f$.
 
Suppose now that $\EE$ is the only equivalence class of $\sim$, and let $E_0, \ldots, E_n$ be an enumeration of $\EE$ such that $E_0 \sim_0 E_1$. Since $\left| Y \right|>1$, the family $\EE$ is not a singleton, that is, $n>0$. Denote $Y_i=Y_{E_i}$. We will find finite sets $W_i \subseteq E_i$ such that for all  $i,i' \leq n$
\begin{enumerate}
\item $Y_i \subseteq W_i$,
\item $W_i$, $W_{i'}$ are similar,
\item $\left|W_i \right| <\left| Y \right|$.
\end{enumerate}

Fix $f\in \Iso(X)$ such that $f(x_0) \in Y_1$ for some $x_0 \in Y_0$. Such $f$ and $x_0$ exist by the assumption that $E_0 \sim_0 E_1$. Put 
\[ W_0= f^{-1}[Y_1] \cup Y_0,\, W_1= f[Y_0] \cup Y_1. \] 
Then, obviously, Condition (1) holds for $W_0, W_1$, and
\[ \left| W_0 \right|, \, \left| W_1 \right| < \left| Y_0 \cup Y_1 \right| \]
because $f(x_0) \in W_1$. In particular, Condition (3) holds as well. We show Condition (2).
Put $U=f^{-1}[Y_1] \setminus Y_0$. Then $W_0=Y_0 \cup U$, and $f$ witnesses that $W_0$ is similar to $W_1$ because $Y_0 \cap U =\emptyset$, and
\[ f[W_0]=f[Y_0 \cup U]=f[Y_0] \cup f[U]=f[Y_0] \cup Y_1=W_1.\]

Suppose now that $n>1$, and that we have already found $W_i$, for some $0<l<n$ and $i \leq l$, that satisfy Conditions (1)-(3), and we have
\begin{equation}
\label{eq1cor}
\left| W_l \right| < \left| Y_0 \cup \ldots \cup Y_l \right|.
\end{equation}

Fix $f \in \Iso(X)$ such that $f[E_l]=E_{l+1}$.  Put $W_{l+1}= f[W_l] \cup Y_{l+1}$. Then, obviously, Condition (1) holds, and
\[ \left| W_{l+1} \right| < \left| Y_0 \cup \ldots \cup Y_{l+1} \right| \]
because
\[ \left| W_{l} \right| < \left| Y_0 \cup \ldots \cup Y_{l} \right|.\]

Put $U=Y_{l+1} \setminus f[W_l]$, and for each $i \leq l$ fix $f_i \in \Iso(X)$ such that $f_i[W_i]=f[W_l]$. This is possible by Condition (2). Put $W'_i=W_i \cup f^{-1}_i[U]$. Then each $W'_i$ is similar to $W_{l+1}$ because  $W_i \cap f_i^{-1}[U] =\emptyset$, and
\[ f_i[W'_i]=f_i[W_i \cup f^{-1}_i[U]]=f_i[W_i] \cup f_i[f_i^{-1}[U]]=f[W_l] \cup U=W_{l+1}.\]

Therefore $W'_0, \ldots, W'_l, W_{l+1}$ satisfy Conditions (1)-(3), and (\ref{eq1cor}) holds for $W_{l+1}$. Thus, we can replace $W_i$ with $W'_i$ for $i \leq l$. This finishes the inductive construction of sets $W_0, \ldots, W_n$.

Now by the inductive assumption and by Condition (3), we can fix finite $Z_i \subseteq E_i$, for $i \leq n$, such that $W_i \subseteq Z_i$, and every $X$-extendable partial isometry of $W_i$ is extendable to some $X$-extendable isometry of $Z_i$. Moreover, by Condition (2), for every $i, i' \leq n$ there exists an $X$-extendable bijection $g:Z_i \rightarrow Z_{i'}$ such that $g[W_i]=W_{i'}$. Put $Z=\bigcup_i Z_i$.

Suppose that $f$ is an $X$-extendable partial isometry of $Y$. Set $i=0$. Define a directed graph $A_i$ on the set $\{ 0, \ldots, n \}$ by introducing an edge $(j,j')$ if there exists $x \in Y_j$ such that $f(x) \in Y_{j'}$. Observe that every vertex in $A_{i}$ has at most one outgoing edge and at most one incoming edge.

If  $i$ is a vertex with no outgoing edge, then there must exist $i' \leq n$ with no incoming edge because the number of outgoing edges is always equal to the number of incoming edges. As $Z_i$, $Z_{i'}$ are similar, we can find an $X$-extendable isometry $h_i:Z_i \rightarrow Z_{i'}$. Define also $A_{i+1}=A_i \cup (i,i')$. Observe that every vertex in $A_{i+1}$ has at most one outgoing edge and at most one incoming edge as well.

Suppose now that $(i,i') \in A_i$ for some $i' \leq n$ . Fix an $X$-extendable bijection $g:Z_i \rightarrow Z_{i'}$ such that $g[W_i]=W_{i'}$. Then $g^{-1}f \upharpoonright Y_i$ is an $X$-extendable partial isometry of $W_i$, so we can extend it to some $h \in \EIso(Z_i)$. Then $h_i=gh$ is an $X$-extendable bijection between $Z_i$ and $Z_{i'}$ extending $f \upharpoonright Y_i$. Finally, put $A_{i+1}=A_i$.

In this way we can inductively define $X$-extandable isometric bijections $h_i$ and graphs $A_i$ for $i \leq n$. Since every vertex in $A_n$ has exactly one outgoing edge and at most one incoming edge, actually every vertex in $A_n$ has exactly one incoming edge as well. Therefore the mapping $\bigcup_i h_i$ is an $X$-extendable bijection between $Z$ and $Z$, that is, it is an $X$-extendable isometry of $Z$ extending $f$


\end{proof}





%


\section{Main results}

\begin{lemma}
\label{le:max}
Let $X$ be an ultrametric space, and let $\BB$ be an FD family.
\begin{enumerate}
\item  Every ball which is $\BB$-non-extremal in $X$, is contained in a ball which is r-maximal $\BB$-non-extremal in $X$. Moreover,  $\mathcal{O}^\BB(D) \subseteq [D]$ for every  ball $D$ which is r-maximal $\BB$-non-extremal in $X$,
\item if $C,C'$ are balls such that $C' \in [C]$, and none of $C,C'$ contains an element of $\BB$, then $C' \in \mathcal{O}^\BB(C)$,
\item if $C$ is a ball which is  r-maximal $\BB$-non-extremal in $X$, and $\BB'$ is an FD family strengthening $\BB$, and such that none of elements of $[C]$ contains an element of $\BB'$, then $\mathcal{O}^{\BB'}(C)=[C]$, and $\tau(Id,\BB')$ acts on $[C]$ as a full symmetric group.
%
%
\end{enumerate}
\end{lemma}

\begin{proof}
We prove Point (1). 
Let $C$ be a ball which is $\BB$-non-extremal in $X$. Let $r=\diam(\mathcal{O}^\BB(C))$, and let $D$ be the unique ball of radius $r$ that contains $C$. Clearly, $D$ is $\BB$-non-extremal in $X$. Suppose that $D'$ is a ball of radius $r'>r$ such that $D \subseteq D'$, and there exists $f \in \tau(Id,\BB)$ such that $f[D]  \neq D'$. But then $\dist(D',f[D']) \geq r'$, so $\dist(C,f[C])>r$, which is a contradiction. Using a similar argument, one can show that $\mathcal{O}^\BB(D) \subseteq [D]$ for every ball $D$, which is r-maximal $\BB$-non-extremal in $X$.

Point (2) directly follows from (\ref{eq0}). Point (3) follows from Points (1) and (2).
\end{proof}

\begin{lemma}
\label{le:agree}
Let $X$ be an ultrametric space, let $\mathcal{B}$ be an FD family such that $X$ has the $\mathcal{B}$-extremality property, and let $Z \subseteq X$ be a finite set intersecting every element of $\BB$.  There exists an FD family $\BB'$ such that 
\begin{enumerate}
\item  $\BB'$ strengthens $\BB$,
\item $X$ has the $\BB'$-extremality property,
\item  every $g \in \Iso(X)$ permuting $Z$ also permutes $\BB'$.
\end{enumerate}
\end{lemma}

\begin{proof}
Let 
\[ \BB''=\{ g[B]: B \in \BB, \, g \in \Iso(X) \mbox{ permutes } Z   \}. \]
Note that $\BB''$ is finite because $Z$ intersects every element of $\BB$.

Let $\BB' \subseteq \BB''$ be a subfamily of $\BB''$ consisting of all balls which do not contain any strictly smaller ball in $\BB''$. Clearly, $\BB'$ is an FD family strengthening $\BB$. First we show that every $g \in \Iso(X)$ permuting $Z$ also permutes $\BB'$. Suppose that $g[B] \not \in \BB'$ for some $g \in \Iso(X)$ permuting $Z$, and $B \in \BB'$. Then there exists $C \in \BB''$ such that $C \subsetneq g[B] $. But then $g^{-1}[C] \in \BB''$, and $g^{-1}[C] \subsetneq B$, contradicting the assumption that $B \in \BB'$.

We prove now that $X$ has the $\BB'$-extremality property. Suppose it does not, and let $C$ be a ball which is r-maximal $\BB'$-non-extremal in $X$. Such $C$ exists, and $\mathcal{O}^{\BB'}(C) \subseteq [C]$ by Point (1) of Lemma \ref{le:max}. Since $C$ is $\BB'$-non-extremal in $X$, and $\BB'$ is finite, Points (1) and (2) of  Lemma \ref{le:max} imply that $[C]$ must be finite and non-trivial.

Clearly, none of the elements of $\mathcal{O}^{\BB'}(C)$ contains an element of $\BB'$. If there exists $B' \in \BB'$ and $C' \in \mathcal{O}^{\BB'(C)}$ such that $C' \subsetneq B'$, then $B'$ has radius strictly larger than $C'$. Therefore $C' \subsetneq B'$ for every $C' \in [C]$, and there exists $g \in \Iso(X)$ permuting $Z$ such that $g[B'] \in \BB$. But then $C' \subsetneq g[B']$ for each $C' \in [g[C]]$, that is, 
\[ \mathcal{O}^\BB(g[C])=\mathcal{O}^{\{g[B]\}}(g[C])=[g[C]] \]
is finite and non-trivial, which means that $X$ does not have the $\BB$-extremality property; a contradiction. Therefore every element of $\mathcal{O}^{\BB'}(C)$ is disjoint from every element of  $\BB'$.


Now fix  $B' \in \mathcal{B}'$ such that  
\begin{equation}
\label{eq:agree:1}
\dist(C,B')=\dist(C,\mathcal{B}').
\end{equation}
Observe that $C$ is $\{B'\}$-non-extremal.  Fix $g \in \Iso(X)$ permuting $Z$ and such that $g[B'] \in \BB$, and observe that, since $\mathcal{B}'$ is closed under permutations of $Z$, none of $g[C']$, for $C' \in \mathcal{O}^{\mathcal{B}'}(C)$, contains an element of $\mathcal{B}$. By Point (2) of Lemma \ref{le:max}, $\mathcal{O}^{\mathcal{B}}(g[C])$ is not a singleton.

However, if $C$ is $\{B'\}$-non-extremal, we also have that $g[C]$ is $\{g[B'] \}$-non-extremal. But $\mathcal{O}^{\mathcal{B}}(g[C]) \subseteq \mathcal{O}^{\{ g[B'] \}}(g[C])$, which is impossible because the former set is infinite while the latter is finite. Thus, $X$ has the $\BB'$-extremality property.

\end{proof}  

\begin{lemma}
\label{le:one:0}
Let $X$ be an ultrametric space, which has the extremality property, and let $Z \subseteq X$ be finite. There exists $h \in \Iso(X)$ such that 
\begin{enumerate}
\item $d(x,h(x'))=d(g(x),hg(x'))$;
\item $d(x,h(x'))=d(g(x),h(x'))$;
\item  $h(x) \in Z$ implies that $h(x)=x$
\end{enumerate}
for all $x,x' \in Z$ and $g \in \Iso(X)$ permuting $Z$.
\end{lemma}

We start with a lemma which takes care of the simplest case of Lemma \ref{le:one:0}.

\begin{lemma}
\label{le:one:aux}
Let  $X$ be an ultrametric space, and let $Z \subseteq X$ be finite. If there exists a ball $B $ in $X$ such that $Z \subseteq $B, and $\left| \mathcal{O}(B) \right|>1$, then Lemma \ref{le:one:0} holds for $X$ and $Z$.
\end{lemma}

\begin{proof}
Fix $B$ as in the statement of the lemma. Since $\left|\mathcal{O}(B)\right|>1$, there exists $h \in \Iso(X)$ such that $h[B]$ is disjoint from $B$. We show that $h$ is as required.

Put $r=\dist(B,h[B])$. Clearly, $r>0$. Fix $x, x' \in Z$ and $g \in \Iso(X)$ permuting $Z$. Then
\[d(x,h(x'))=d(g(x),hg(x'))=\dist(g(x),h(x'))=r \]
because the assumption that $g \in \Iso(X)$ permutes $Z$ implies that $g(x), g(x') \in B$, while $h(x), h(x'), hg(x') \in h[B]$. Therefore Points (1)-(3) of  Lemma \ref{le:one:0} hold for $x,x'$ and $g$.
\end{proof}

\begin{proof}[Proof of Lemma \ref{le:one:0}]
We proceed by induction on the size of $Z$. Fix a finite $Z \subseteq X$. 
If $Z$ is a singleton, then $h=Id$ is, clearly, as required. 
Suppose now that the lemma holds for sets of size $< \left| Z \right|$. If there exists a ball $B$ as in Lemma \ref{le:one:aux}, then we are done.
Otherwise $\left| \mathcal{O}(B) \right|=1$ for every ball $B$ such that $Z \subseteq B$.  Let $r=\diam(Z)$, and let $\EE$ be an FD family consisting of all balls of radius $r$, which have non-empty intersection with $Z$. Note that 
\begin{equation}
\label{eq:one:1}
\dist(E,E')=r>0 
\end{equation}
for every $E,E' \in \EE$ with $E \neq E'$. Put $Z_E=Z \cap E$ for $E \in \EE$.

\medskip

\emph{Claim}. $\mathcal{O}(E)=[E]$ for $E \in \EE$. Moreover, if $E \in \EE$, and $[E]$ is finite, then $[E]=\{E\}$, and $E$ has the extremality property.

\medskip

\emph{Proof of the claim.} Since every ball of radius $>r$ that contains $E$ must contain $Z$ as well, we have that $\left| \mathcal{O}(B) \right|=1$ for every such ball $B$. Therefore the same argument as in the proof of Point (1) of Lemma \ref{le:max} shows that $\mathcal{O}(E)=[E]$. Since $X$ has the extremality property, either $[E]$ is infinite or $[E]=\{E\}$. In the latter case, if a ball $B \subseteq E$ is non-extremal in $E$, then it is non-extremal in $X$ as well. But $X$ has the extremality property, so $E$ also must have the extremality property. This finishes the proof of the claim.
 
\medskip

Define
\[ \FF =\{ E \in \EE:  [E] \mbox{ is infinite} \}. \]
Then for every $E \in \FF$ we can find $F_E \in [E]$ so that
\begin{enumerate}[(a)]
\item $F_E \cap E' = \emptyset$ for $E' \in \EE$;
\item $F_E \cap F_{E'} = \emptyset$ for every $E' \in \FF$ with $E \neq E'$.
\end{enumerate}

In particular, 
\begin{enumerate}[(a)]
\setcounter{enumi}{2}
\item $\dist(E,F_{E'})=r$ for every $E \in \EE$, $E' \in \FF$.
\end{enumerate}

Fix $E \in \EE$. If $E \in \FF$, then choose an $X$-extendable bijection $h_E:E \rightarrow F_E$. Otherwise, by the claim, $E$ has the extremality property. Also, $\left| Z_E \right|< \left| Z \right|$ because $\left| \EE \right|>1$. 
Thus, by the inductive assumption, we can find $h_E \in \Iso(E)$ witnessing that Lemma \ref{le:one:0} holds for $E$ and $Z_E$.

Clearly, $\bigcup_{E \in \EE} h_E$ is a bijection onto its image. Fix $x \in E$, $y \in E'$, where $E,E' \in \EE$, $E \neq E'$. Then $d(x,y)=r$ by  (\ref{eq:one:1}). If $E, E' \in \FF$, then $h(x) \in F_E$, $h(y) \in F_{E'}$, that is, $d(h_E(x),h_{E'}(y))=r$ by Points (b) and (c). Similarly, if $E \in \FF$, $E' \not \in \FF$, then $h(x) \in F_E$, $h(y) \in E'$, so $d(h(x), h(y))=r$ by Point (c). Finally, if $E, E' \not \in \FF$, then $h(x) \in E$, $h(y) \in E'$, and $d(h(x),h(y))=r$. Thus, $\bigcup_{E \in \EE} h_E$ is isometric. Since every $h_E$ is $X$-extendable, $\bigcup_{E \in \EE} h_E$ is also $X$-extendable, so we can extend it to an isometry $h \in \Iso(X)$. 

We verify that $h$ satisfies Point  (1). Fix $x, x' \in Z$, and $g \in \Iso(X)$ permuting $Z$. Let $E,E' \in \EE$ be such that $x \in E$, $x' \in E'$. 

\emph{Case 1:} $E' \in \FF$. Then $g[E'] \in \FF$, so by Point (c), 
\[ d(x,h(x'))=\dist(E,h[E'])=\dist(E,F_{E'})=r, \]
\[ d(g(x),hg(x'))=\dist(g[E],hg[E'])=\dist(g[E],F_{g[E']})=r. \]
%
%

\emph{Case 2:} $E' \not \in \FF$, $E \neq E'$. Then $g[E'] \not \in \FF$, so, by the claim, $h[E']=E'$, $hg[E']=g[E']$, and
\[ d(x,h(x'))=d(E,h[E'])=d(E,E')=r, \]
\[ d(g(x),hg(x'))=\dist(g[E],hg[E'])=\dist(g[E],g[E'])=r. \]

\emph{Case 3:} $E' \not \in \FF$, $E=E'$. Then, by the claim,  $E$ has the extremality property, and $[E]=\{E\}$, that is, $g[Z_E]=Z_E$. Therefore
\[ d(x,h(x'))=d(g(x),hg(x')) \]
by the inductive assumption.

Point (2) can be shown similarly to Point (1):

\emph{Case 1:} $E' \in \FF$. Then $\dist(g[E],F_{E'})=r$ by Point (a) and because $g$ permutes $\EE$. Thus, we get
\[ d(g(x),h(x'))=\dist(g[E], F_{E'})=r.\] 

\emph{Case 2:} $E' \not \in \FF$. Observe that in this case $g[E] \neq E'$, and $h[E']=E'$. Therefore
\[ d(g(x),h(x'))=\dist(g[E], E')=r.\] 

To prove that Point (3) holds, fix $x \in Z$, and $E \in \EE$ such that $x \in E$. If $E \in \FF$, then Point (a) implies that $h(x) \not \in Z$. Otherwise, $h(x)=h_E(x) \in E$, so we can use the inductive assumption.
\end{proof}

Now we consolidate Points (1) and (2) of Lemma \ref{le:one:0} into one statement.

\begin{lemma}
\label{le:one}
Let $X$ be an ultrametric space which has the extremality property, and let $Z \subseteq X$ be finite. There exists $h \in \Iso(X)$ such that
\begin{enumerate}
\item $d(x,h(x'))=d(g(x),hg'(x'))$,
\item $h(x) \in Z$ implies that $h(x)=x$
\end{enumerate}
for all $x,x' \in Z$ and $g,g' \in \Iso(X)$ permuting $Z$.
\end{lemma}

\begin{proof}
For $x, x' \in Z$, $g,g' \in \Iso(X)$  permuting $Z$, and $h \in \Iso(X)$ as in the statement of Lemma \ref{le:one:0}, we have
\[ d(g(x),hg'(x'))=d((g')^{-1}g(x),h(g')^{-1}g'(x'))= \]
\[ =d((g')^{-1}g(x),h(x')=d(x,h(x')). \]

The first equality follows from Point (1) of Lemma \ref{le:one:0}, and the last one follows from Point (2) of this lemma.

\end{proof}

The following corollary is straightforward in light of Lemma \ref{le:one}.

\begin{corollary}
\label{co:one:1}
Let $X$ be an ultrametric space. For every ball $B$ in $X$ which has the extremality property, and every finite $Z \subseteq B$ we can fix $h=h_{B,Z}$ as in Lemma \ref{le:one} applied to $B$ and $Z$ so that the following holds. For every ball $B$ in $X$, which has the extremality property, every finite $Z \subseteq B$, and $g \in \Iso(X)$ there exists $f \in \Iso(X)$ such that  $h_{g[B],g[Z]}=f h_{B,Z} f^{-1}$.
\end{corollary}

\begin{proof}
Fix a ball $B$, finite $Z \subseteq B$, and $h_{B,Z}=h$ given by Lemma \ref{le:one}. For every $B'$ similar to $B$, and $Z' \subseteq B'$ such that $B \neq B'$ or $Z \neq Z'$, and there exists $f \in \Iso(X)$ such that $f[Z]=Z'$, fix such $f$, and put $h_{B',Z'}=fh_{B,Z}f^{-1}$.
\end{proof}

\begin{lemma}
\label{le:one:1}
Let $X$ be an ultrametric space, let $Z \subseteq X$ be finite, and let $\mathcal{B}$ be an FD family such that $X$ has the $\mathcal{B}$-extremality property. There exists $h \in \tau(Id,\mathcal{B})$ such that
\begin{enumerate}
\item $d(x,h(x'))=d(g(x),hg'(x'))$,
\item $h(x) \in Z$ implies that $h(x)=x$
\end{enumerate}
 for all $x,x' \in Z$, and $g,g' \in \Iso(X)$ permuting $Z$ and $\BB$, and such that $\tau(g,\mathcal{B})= \tau(g',\mathcal{B})$.
 \end{lemma}

\begin{proof}
Fix $x \in Z$. If there exists $B \in \mathcal{B}$ such that  $x \in B$, then define $E_x=B$. Otherwise, let $E_x$ be the unique  ball such that $x \in E_x$, and radius of $E_x$ is equal to $\dist(x,\mathcal{B})$. Let $\EE$ be the collection of all the balls $E_x$, $x \in Z$, of radii $r_{E_x}$. Define $Z_E=Z \cap E$ for $E \in \EE$. 
Observe that $\EE$ is pairwise disjoint, and that $ \mathcal{O}^\mathcal{B}(E) \subseteq [E]$ for $E \in \EE $.

Define
\[ \FF=\{ E \in \EE: \mathcal{O}^\mathcal{B}(E) \mbox{ is infinite } \}. \]

Then for every $E \in \FF$ we can find $F_E \in [E]$ disjoint from every element of $\BB$, and such that 
\begin{equation}
\label{eq:le:1:1}
F_{E} \cap E'=\emptyset, \, F_E \cap F_{E''} = \emptyset
\end{equation}
for every $E' \in \EE$, $E'' \in \FF$ with $E \neq E''$. In particular,
\begin{equation}
\label{eq:le:1:2}
\dist(E, F_{E'})=\dist(F_E,F_{E''})=r_E, \mbox{ if } E',E'' \in [E] \mbox{ and }  E'' \neq E.
\end{equation}

Fix $E \in \EE$. If $E \in \FF$, then choose an $X$-extendable isometric bijection $h_E:E \rightarrow F_E$. Otherwise by the $\BB$-extremality of $X$, $\mathcal{O}^\mathcal{B}(E)=\{E\}$, so $E$ has the extremality property. Therefore we can apply Corollary \ref{co:one:1} to $E$ and $Z_E$ to find $h_E=h_{E,Z_E} \in \Iso(E)$. We check that the mapping $\bigcup_{E \in \EE} h_E$ can be extended to an isometry $h \in \tau(Id,\BB)$ exactly in the same way as we did it in the proof of Lemma \ref{le:one:0} for $\bigcup_{E \in \EE} h_E$ defined there.

We verify that Point (1) holds for $h$. Fix $x,x' \in Z$ and $E, E' \in \EE$ such that $x \in E$, $x' \in E'$.
Fix $g,g' \in \Iso(X)$ permuting $Z$ and $\BB$, and such that $\tau(g,\mathcal{B})=\tau(g',\mathcal{B})$. Fix $B,B' \in \mathcal{B}$ such that 
\[ \dist(E,B)=\dist(E, \BB), \,  \dist(E',B')=\dist(E',\BB). \]

\emph{Case 1:} $E=E'$, $E \in \FF$. Then $g[E] \in \FF$,  and $g[B]$ witnesses that $r_E=r_{g[E]}$. Also, $g'[E] \in [g[E]]$ because $g[B]=g'[B]$. By (\ref{eq:le:1:2}),
\[ d(x,h(x'))=\dist(E,h[E])=\dist(E,F_E)=r_E, \]
\[ d(g(x),hg'(x'))=\dist(g[E],hg'[E])=\dist(g[E],F_{g'[E]})=r_{g[E]}=r_E.    \]

\emph{Case 2:} $E=E'$, $E \not \in \FF$. Because $\mathcal{O}^\BB(E)=\{E\}$, we have 
\[ g[E]=g'[E], \,  g[Z_E]=g'[Z_E], \]
and 
\[ h_E=h_{E,Z_E},\,  h_{g[E]}=h_{g'[E]}=h_{g[E],g[Z_E]}.\]
Fix $f \in \Iso(X)$ such that 
\[ f[Z_E]=g[Z_E], \, h_{g[E]}=f h_{E} f^{-1}.\]
Because
\[ f^{-1}g[E]=f^{-1}g'[E]=E, \] 
and both $f^{-1}g$ and $f^{-1}g'$ permute $Z_E$, Lemma \ref{le:one} implies that
\[ d(g(x),hg'(x'))=d(g(x),h_{g[E]}g'(x'))=d(g(x),f h_{E} f^{-1} g'(x'))= \]
\[ d(f^{-1}g(x),h_{E} f^{-1} g'(x'))=d(x,h_{E}(x'))=d(x,h(x')).\]

%
%
%

%
%

\emph{Case 3:} $E \neq E'$, $\dist(E,B), \dist(E',B')< \dist(B,B')$. Then
\[ d(x,h(x'))=d(g(x),hg'(x'))=\dist(B,B')\]
by the assumption that $g$ and $g'$ permute $\BB$ in the same way, that is, $g[B]=g'[B]$, $g[B']=g'[B']$, and that $h$ setwise fixes $B,B'$.

\emph{Case 4:} $E \neq E'$, $\dist(B,B') \leq \dist(E,B)$ or $\dist(B,B') \leq \dist(E',B')$. Observe that in this case we can assume that $B=B'$. If $\dist(E,B)=\dist(E',B)$, then the radii of $E$ and $E'$, as well as $\dist(E,E')$, are all equal to $\dist(E,B)$, and this case can be dealt with analogously to Case 1. Otherwise, we can assume that $\dist(E,B)>\dist(E',B)$. Since $g'[B']=g[B']$, and $h$ setwise fixes $B,B'$, we have that
\[ d(x,h(x'))=d(g(x),hg'(x'))=\dist(E,B).\]

To verify Point (2), fix $x \in Z$, and $E \in \EE$ such that $x \in E$. If $E \in \FF$, then (\ref{eq:le:1:1}) implies that $h(x) \not \in Z$. Otherwise, $h(x) \in E$. Therefore if $h(x) \in Z_E$, then $h(x)=h_E(x)=x$ by Point (2) of Lemma \ref{le:one}.
\end{proof}

Now we are ready to prove the main technical lemma.

\begin{lemma}
\label{le:main}
Let $X$ be a countable ultrametric space, let $n \in \NN$, and let $\mathcal{A}$ be an FD family in $X$. 
There exist $\gamma_i \in \tau(Id,\mathcal{A})$, $i \leq n$, such that the following holds. If  $\mathcal{B} $ is an FD family such that $X$ has the $\mathcal{B}$-extremality property, then the orbit of $(\gamma_0, \ldots, \gamma_n)$ under the diagonal conjugation action of $\tau(Id,\BB)$ on $(\tau(Id,\mathcal{A}))^{n+1}$ is somewhere dense.

Moreover, if $X$ has the $\mathcal{A}$-extremality property, then the orbit of $(\gamma_0, \ldots, \gamma_n)$ under the action of $\tau(Id,\mathcal{A})$ is dense in $(\tau(Id,\mathcal{A}))^{n+1}$. 
\end{lemma}

\begin{proof}
If there is no FD family $\BB$ such that $X$ has the $\BB$-extremality property, there is nothing to prove. Otherwise, let $(f_{0,k}, \ldots, f_{n,k}, Z_k, \mathcal{B}_k)$, $k \in \NN$, be a list such that every tuple $(f_0, \ldots, f_n,Z,\BB)$ satisfying
\begin{enumerate}[(1)]
\item $Z \subseteq X$ is finite,
\item each $f_i$ is an $X$-extendable mapping defined on $Z$,
\item  $\BB$ is an FD family such that $X$ has the $\BB$-extremality property
\end{enumerate} 
appears on it infinitely many times. 

Let $X=\{x_k\}_{k \in \NN}$. We define finite $X_k \subseteq X$, $\gamma_{i,k} \in \EIso(X_k)$, $i \leq n$, and $h_k \in \tau(Id,\BB_k)$  such that for every $k \in \NN$ and $i \leq n$
\begin{enumerate}[(a)]
\item $x_k \in X_k$, $X_k \subseteq X_{k+1}$;
\item $\gamma_{i,k} \subseteq \gamma_{i,k+1}$;
\item if $Z_k \subseteq X_{k}$, $f_{i,k}[Z_k] \subseteq X_k$, and $f_{i,k}$, $\gamma_{i,k}$ induce the same permutation of $\BB_{k}$ for $i \leq n$, then 
\[ h^{-1}_{k+1} \gamma_{i,k+1}  h_{k+1}(x)=f_{i,k}(x) \mbox{ for } i \leq n \mbox{ and } x \in Z_k.\]
\end{enumerate}

Fix a finite $X_0 \subseteq X$ such that $x_0 \in X_0$, and $X_0$ has non-empty intersection with each element of $\mathcal{A}$. Define $\gamma_{i,0}=Id \upharpoonright X_0$ for $i \leq n$, and $h_0=Id$. 

Suppose that $X_k$, $\gamma_{i,k}$, and $h_k$ have been constructed for some $k$. We show how to construct  $X_{k+1}$, $\gamma_{i,k+1}$, and $h_{k+1}$.

If  the assumptions of Condition (c) are not satisfied, then, applying Lemma \ref{le:ext1}, find a finite $X_{k+1} \subseteq X$ such that $X_k \cup \{x_{k+1}\} \subseteq X_{k+1}$, and every $X$-extendable partial isometry of $X_k \cup \{x_{k+1}\}$ can be extended to an $X$-extendable isometry of $X_{k+1}$. In particular, each $\gamma_{i,k}$ can be extended to some $\gamma_{i,k+1} \in \EIso(X_{k+1})$. Also, put $h_{k+1}=Id$. Clearly, Conditions (a) and (b) hold.
 

Otherwise, applying Lemma \ref{le:ext1} find a finite $X'_{k+1} \subseteq X$ such that $X_k \cup \{x_{k+1}\} \subseteq X'_{k+1}$, and every $X$-extendable partial isometry of $X_k \cup \{x_{k+1}\}$ can be extended to an $X$-extendable isometry of $X'_{k+1}$. In particular, each $\gamma_{i,k}$ can be extended to some $\gamma'_{i,k+1} \in \EIso(X'_{k+1})$, and each $f_{i,k}$ can be extended to some $f'_{i,k} \in \EIso(X'_{k+1})$.

Now  find $h \in \tau(Id,\BB)$ applying Lemma \ref{le:one:1} to $\BB=\BB_k$ and $Z=X'_{k+1}$.
We show that each $\gamma'_{i,k+1} \cup h f'_{i,k} h^{-1}$ is an $X$-extendable isometry of $X'_{k+1} \cup h[X'_{k+1}]$. Since  $\gamma'_{i,k+1} \in \EIso(X'_{k+1})$, and $h f'_{i,k} h^{-1} \in \EIso(h[X'_{k+1}])$, Lemma \ref{le:ext0} implies that it suffices to show that the union of these mappings is isometric.

First of all, observe that it is well defined. Suppose that $h(x) \in X'_{k+1}$ for some $x \in X'_{k+1}$.  Because both $\gamma'_{i,k+1}$ and $f'_{i,k}$ are $X$-extendable (so we can regard them as isometries of $X$), they permute $X'_{k+1}$, $\BB_k$,  and they induce the same permutation of $\BB_k$ for each $i \leq n$ (see Point (c)), we can apply Lemma \ref{le:one:1}. By Point (2) of this lemma, $h(x)=x$, so
\[ d(\gamma'_{i,k+1}(h(x)), h f'_{i,k} h^{-1}(h(x)))=d(\gamma'_{i,k+1}(x),h f'_{i,k}(x))=d(x,x), \]
where the last equality follows from Point (1) of Lemma \ref{le:one:1}. In other words,
\[ \gamma'_{i,k+1}(h(x))=h f'_{i,k} h^{-1}(h(x)). \]

In order to see that it is a distance-preserving mapping, fix $x,x' \in X'_{k+1}$. Then $h(x') \in h[X'_{k+1}]$, and 
\[ d(\gamma'_{i,k+1}(x),h f'_{i,k} h^{-1}(h(x')))=d(\gamma'_{i,k+1}(x),h f'_{i,k}(x'))=d(x,h(x')), \]
where the last equality follows from Point (1) of Lemma \ref{le:one:1}.

Applying Lemma  \ref{le:ext1} again, find a finite  $X_{k+1} \subseteq X$ such that $X'_{k+1} \cup h[X'_{k+1}] \subseteq X_{k+1}$, and each $X$-extendable partial isometry of $X'_{k+1}\cup h[X'_{k+1}]$ can be extended to an $X$-extendable isometry of $X_{k+1}$. In particular, each $\gamma'_{i,k+1} \cup h f'_{i,k} h^{-1}$ can be extended to some  $\gamma_{i,k+1} \in \EIso(X_{k+1})$. Put $h_{k+1}=h$. Then for every $i \leq n$ and $x \in Z_k$
\[ h^{-1}_{k+1} \gamma_{i,k+1} h_{k+1}(x)=h^{-1}_{k+1} h_{k+1} f'_{i,k} h^{-1}_{k+1} h_{k+1}(x)=f'_{i,k}(x)=f_{i,k}(x), \]
so Condition (c) holds.

This completes the inductive construction. Put $\gamma_i= \bigcup_k \gamma_{i,k}$ for $i \leq n$. By Conditions (a) and (b), each $\gamma_i$ is an isometry of $X$, and the step $k=0$ of the construction guarantees that $\gamma_i \in \tau(Id, \mathcal{A})$.

We will show that $\gamma_i$ are as required. Fix an FD family $\BB$ such that $X$ has the $\BB$-extremality property, and fix $k$ such that $X_k$ intersects every element of $\BB$. By Lemma \ref{le:agree}, there exists an FD family $\BB'$ such that $X$ has the $\BB'$-extremality property,  $\BB'$ strengthens $\BB$, and each $\gamma_i$ permutes $\BB'$ (because each $\gamma_i$ permutes $X_k$.)

Fix $g_i \in \tau(\gamma_i, \BB')$ for $i \leq n$. Obviously, each $g_i$ induces the same permutation of $\BB'$ as $\gamma_i$ does. We will show that for every finite $Z \subseteq X$ intersecting each element of $\BB'$ there exists $h \in \tau(Id,\BB')$ such that
\[ h^{-1} \gamma_i  h(x)=g_i(x) \]
for $x \in Z$, and $i \leq n$. As $g_i$ and $Z$ are arbitrary, this implies that the orbit of $(\gamma_0, \ldots, \gamma_n)$ under the action of $\tau(Id,\BB')$ is dense in $\tau(\gamma_0, \BB') \times \ldots \times \tau(\gamma_n, \BB')$.


Fix such $Z \subseteq X$, and let $f_i=g_i \upharpoonright Z$, $i \leq n$. Since there exist infinitely many $k$ with $(f_{0,k}, \ldots,f_{n,k},Z_k, \mathcal{B}_k)=(f_0, \ldots,f_n,Z,\BB')$, we can fix such $k$ so that, additionally, $Z, f_i[Z] \subseteq X_k$ for $i \leq n$. Because $Z$ intersects every element of $\BB'$, $f_i$, $\gamma_{i,k}$ induce the same permutation of $\BB'$ for $i \leq n$. 
Therefore Condition (c) yields that
\[ h^{-1}_{k+1} \gamma_i h_{k+1}(x)= h^{-1}_{k+1} \gamma_{i,k+1} h_{k+1}(x)=f_{i}(x)=g_i(x) \]
%
%
 %
for $x \in Z$, and $i \leq n$.

The `moreover' part follows from the fact that $\tau(\gamma_i,\mathcal{A})=\tau(Id,\mathcal{A})$ for $i \leq n$, and thus the above argument works for $\BB'=\mathcal{A}$, provided that $X$ has the $\mathcal{A}$-extremality property.
\end{proof}
 
\begin{theorem}
\label{th:2Con}
Let $X$ be an ultrametric space. Suppose that $\BB$ is an FD family that cannot be strengthened to an FD family $\BB'$ such that $X$ has the $\BB'$-extremality property. Then one of the following  holds:

\begin{enumerate}
\item there exists an FD family $\BB'$ strengthening $\BB$, and such that there are infinitely many balls, which are r-maximal $\BB'$-non-extremal in $X$;
\item there exists an infinite sequence $C_0 \supsetneq C_1 \supsetneq \ldots$ of   balls such that $C_{k+1}$ is r-maximal non-extremal in $C_k$, and $C_0$ is r-maximal $\BB$-non-extremal in $X$.
\end{enumerate}
\end{theorem}

\begin{proof}
Suppose that Point (1) does not hold. For every $k \in \NN$ we will define a family $T_k \subseteq \NN^{k+1}$ of finite sequences of natural numbers of length $k+1$, and a family of balls $\mathcal{C}_k$, indexed by elements of $T_k$, so that for all $k$ and $\sigma \in T_k$

\begin{enumerate}[(a)]
\item $\mathcal{C}_0$ is the family of all balls, which are r-maximal $\BB$-non-extremal in $X$,
\item $T_k$ is finite and non-empty,
\item  if $k>0$, and $\sigma=\sigma'^\frown i$, then $\sigma' \in T_{k-1}$,
\item $C \subseteq C_\sigma$ is r-maximal non-extremal in $C_\sigma$  iff $C=C_{\sigma^\frown i}$ for some $\sigma^\frown i \in T_{k+1}$.
\item for $\BB_k=\BB \cup \mathcal{C}_k$, every r-maximal $\BB_k$-non-extremal ball $C$ in $X$ is contained in $C_\sigma$ for some $\sigma \in T_k$, and $C$ is r-maximal non-extremal in this $C_\sigma$.
\end{enumerate}

By our assumption there are only finitely many r-maximal $\BB$-non-extremal  balls $C_{(0)}, \ldots C_{(n)}$ in $X$. Put $T_0=\{ (0), \ldots, (n) \}$, $\mathcal{C}_0=\{C_{(0)}, \ldots C_{(n)}\}$, and suppose that $T_l$, and $\mathcal{C}_l$ have been defined for all $l \leq k$. We will construct $T_{k+1}$, $\mathcal{C}_{k+1}$.


Fix $\sigma \in T_k$, and let $\mathcal{C}_\sigma$ be the collection of all balls which are contained in $C_\sigma$, and are r-maximal $\BB_k$-non-extremal in $X$. We can assume that $\mathcal{C}_\sigma$ is indexed by elements of the form $\sigma^\frown n$, where $n$ is a natural number. Let $T_\sigma$ be the collection of these indices. Finally, let $T_{k+1}= \bigcup_{\sigma \in T_k} T_\sigma $, $\mathcal{C}_{k+1}= \bigcup_{\sigma \in T_k} \mathcal{C}_\sigma $.

Clearly, Conditions (c) and (d) are satisfied. To see that Condition (e) holds as well, notice that if $C$ is r-maximal $\BB_{k+1}$-non-extremal, then either it is strictly contained in an element of some $\mathcal{C}_\sigma$ or, by Point (1) of Lemma \ref{le:max}, it is also r-maximal $\BB_k$-non-extremal. If the former is true, then $C$ is r-maximal non-extremal in $C_\sigma$. But by the inductive assumption, $\mathcal{C}_{k+1}$ is the collection of all r-maximal $\BB_k$-non-extremal balls, so the latter is not possible.

Because $X$ does not have the $\BB_k$-extremality property, Lemma \ref{le:max} implies that there exists an r-maximal $\BB_k$-non-extremal ball in $X$. By Point (e), it is a member of $\mathcal{C}_{k+1}$. Because Point (1) does not hold, each $\mathcal{C}_\sigma$ is finite. Therefore Point (b) holds for $T_{k+1}$.

Put $T=\bigcup_k T_k$. By Condition (c), $T$ is a  tree. By Condition (b), $T$ is infinite, and finitely branching, so K\"{o}nig's lemma implies that there exists an infinite branch $S$ in $T$. For $k \in \NN$ put $C_k=C_\sigma$, where $\sigma$ is the unique sequence in $S$ of length $k+1$. By Conditions (a) and (d), the sequence $C_0,C_1, \ldots$ witnesses that Point (2) holds.
 \end{proof}

\begin{theorem}
\label{th:main:2}
Let $X$ be a Polish ultrametric space. Suppose that there exists an FD family $\BB$ that cannot be strengthened to an FD family $\BB'$ such that $X$ has the $\BB'$-extremality property. Then no open subgroup $V \leq \Iso(X)$ has a comeager conjugacy class.
\end{theorem}

\begin{proof}
Fix an open $V \leq \Iso(X)$. Fix an FD family $\BB$ that cannot be strengthened to an FD family $\BB'$ such that $X$ has the $\BB'$-extremality property. Strengthening $\BB$, if necessary, we may assume that $\tau(Id, \BB) \leq V$.
By Theorem \ref{th:2Con}, it suffices to consider two cases.

\emph{Case 1:} Point (1) of Theorem \ref{th:2Con} holds. Let $\BB'$ be an FD family strengthening $\BB$, and such that there exists an infinite family $\mathcal{C}$ of balls that are r-maximal $\BB'$-non-extremal in $X$. Set $W=\tau(Id,\BB')$. We will show that for every $\gamma \in W$ the orbit of $\gamma$ under the conjugation action of $W$ on itself is nowhere dense. By Point (2) of Lemma \ref{le:fo}, this proves that the conjugation action of $V$ on itself has no comeager orbit.


Let $\mathcal{C}_k$, $k \in \NN$, be an enumeration of all $\mathcal{O}^{\BB'}(C)$, $C \in \mathcal{C}$. 
Fix $\gamma \in W$. 
Fix $g \in W$, and an FD family $\BB''$ which strengthens $\BB'$. As $\BB''$ is finite, while $\mathcal{C}$ is infinite and pairwise disjoint, there exists $k_0$ such that no element of $\BB''$ is contained in some element of $\mathcal{C}_{k_0}$. For $g \in W$ denote by $\phi_g$ the permutation of $\mathcal{C}_{k_0}$ induced by $g $. By Point (3) of Lemma \ref{le:max}, every permutation of $\mathcal{C}_{k_0}$ is induced by some $g' \in \tau(g,\BB'')$. As $\mathcal{C}_{k_0}$ is finite and non-trivial, there exists $g' \in \tau(g,\BB'')$ such that the signature of $\phi_{g'}$ differs from the signature of $\phi_{\gamma}$. Because the signature of $\phi_{h\gamma h^{-1}}$ is the same as the signature of $\phi_\gamma$ for every $h \in W$, we get that $h\gamma h^{-1} \not \in \tau(g',\BB'' \cup \mathcal{C}_{k_0})$ for every $h \in W$. Since $g$ and $\BB''$ are arbitrary, this completes the proof of Case 1. 

 \emph{Case 2:} Point (2) of Theorem \ref{th:2Con} holds. Set $W=\tau(Id, \BB)$. As in Case 1, we will show that for every $\gamma \in W$, the orbit of $\gamma$ under the conjugation action of $W$ on itself is nowhere dense.

Let $C_0 \supsetneq C_1 \supsetneq \ldots$ be an infinite sequence of balls such that each $C_{k+1}$ is r-maximal non-extremal in $C_k$, and $C_0$ is r-maximal $\BB$-non-extremal in $X$. Put $\mathcal{C}_k=\mathcal{O}^{\BB}(C_k)$. Obviously, every $\mathcal{C}_k$ is finite and non-trivial, and every $g \in W$ permutes $\mathcal{C}_k$. 
Fix $\gamma \in W$. Fix $g \in W$, and an FD family $\BB'$.

\emph{Case 2a:} There exists $k_0$ such that $g$ induces a non-trivial permutation of $\mathcal{C}_{k_0}$. Let $D \in \mathcal{C}_{k_0}$ be such that $g[D]\neq D$. As $\BB'$ is finite, and $\left| \mathcal{C}_k \right| \rightarrow \infty$, there exists $k_1>k_0$ and $E \in \mathcal{C}_{k_1}$ such that $E \subseteq D$, and no element of $\BB'$ is contained in $E$. Fix $F \in \mathcal{C}_{k_1+1}$ such that $F \subseteq E$. Then $[F] \subseteq \mathcal{C}_{k_1+1}$, and $[F]$ is non-trivial.  For $g \in W$ denote by $\phi_g$ the permutation of $\mathcal{C}_{k_1+1}$ induced by $g $. 

If the signature of $\phi_\gamma$ differs from the signature of $\phi_g$, then the same argument as in Case 1 completes the proof. Otherwise fix $F' \in [F]$ which is distinct from $F$, and an involution $f \in \Iso(X)$ such that $f[F]=F'$ and $f$ is the identity outside of $F\cup F'$. It is straightforward to check that $gf \in \tau(g,\BB')$. 
Since the signature of $\phi_{gf}$ differs from the signature of $\phi_{g}$, we get that $h\gamma h^{-1} \not \in \tau(gf,\BB' \cup \mathcal{C}_{k_1+1})$ for every $h \in W$.
As $g$ and $\BB'$ are arbitrary, this finishes the proof of Case 2a. 

\emph{Case 2b:} The action of $g$ on every $\mathcal{C}_k$ is trivial. If there exists $k_0$ such that the action of $\gamma$ on $\mathcal{C}_{k_0}$ is not trivial, then the same is true about $h \gamma h^{-1}$ for every $h \in W$. Therefore $h \gamma h^{-1} \not \in \tau(g,\BB' \cup \mathcal{C}_{k_0})$ for every $h \in W$. Otherwise, find $k_0$ and $D \in \mathcal{C}_{k_0}$ such that no element of $\BB'$ is contained in $D$, and fix $E \in \mathcal{C}_{k_0+1}$ such that $E \subseteq D$. As every permutation of $[E]$ is induced by some element of $\tau(g, \BB' \cup \mathcal{C}_{k_0})$, and $[E]$ is non-trivial, there exists $g' \in \tau(g, \BB' \cup \mathcal{C}_{k_0})$ which induces a non-trivial permutation of $\mathcal{C}_{k_0+1}$. Therefore $h\gamma h^{-1} \not \in \tau(g',\BB' \cup \mathcal{C}_{k_0+1})$ for every $h \in W$.
\end{proof}

Note that if $X' \subseteq X$ is a dense subspace of an ultrametric space $X$, then every ball $B$ in $X$ determines a unique (as a set) ball $B \cap X'$ in $X'$, and every ball $B'$ in $X'$ determines a unique ball $\overline{B'}$ in $X$. Therefore in the sequel we will not differentiate between balls in $X$ and balls in $X'$.

\begin{lemma}
\label{le:Count}
Let $X$ be a Polish ultrametric space. There exists a countable, dense subspace $X' \subseteq X$ such that for every $g \in \Iso(X)$, and every FD family $\BB$ in $X$ there exists $g' \in \Iso(X)$ such that $g'[X']=X'$, and $g' \in \tau(g,\BB)$. 
In particular, $\Iso(X')$ is dense in $\Iso(X)$, and $\tau_{X'}(g',\BB)$ is dense in $\tau_X(g',\BB)$ for $g' \in \Iso(X')$, if isometries of $X'$ are identified with their unique extensions to isometries of $X$.
\end{lemma}

\begin{proof}
Fix a countable set $X' \subseteq X$ which is dense in $X$ and such that for every $x \in X'$ the set $\db{x} \cap X'$ is dense in $\db{x}$. We show that $X'$ is as required.

Fix $g \in \Iso(X)$, and an FD family $\BB$. For every $B \in \BB$ there exists $x_B \in X' \cap B$ and $y_B \in g[B] \cap X'$ such that $\db{x_B}=\db{y_B}$.  Using the assumption that $X$ is ultrametric it is easy to check that the mapping $g''$ defined on $\{x_B: B \in \mathcal{B}\}$ by $g''(x_B)=y_B$ preserves distances. By Lemma \ref{le:ext0}, $g''$ $X$-extendable to an isometry $g' \in \Iso(X')$. Denote its unique extension to an isometry of $X$ by $g'$ as well. Then $g'[X']=X'$, and $g'[B]=g[B]$ for every $B \in \BB$.
\end{proof}

\begin{theorem}
\label{th:main}
Let $X$ be a Polish ultrametric space. If every FD family  $\BB$ can be strengthened to an FD family $\BB'$ such that $X$ has the $\BB'$-extremality property, then $\Iso(X)$ has a neighborhood basis at the identity consisting of open subgroups with ample generics. Otherwise, no open subgroup $V \leq \Iso(X)$ has a comeager conjugacy class.
\end{theorem}

\begin{proof}
Suppose that every FD family  $\BB$ in $X$ can be strengthened to an FD family $\BB'$ such that $X$ has the $\BB'$-extremality property. By Lemma \ref{le:Count}, there exists a countable $X' \subseteq X$ such that every FD family  $\BB$ in $X'$ can be strengthened to an FD family $\BB'$ in $X'$ such that $X'$ has the $\BB'$-extremality property. 

Fix an open neighborhood of the identity $U \subseteq \Iso(X')$, and fix an FD family $\mathcal{A}$ such that $\tau_{X'}(Id,\mathcal{A}) \subseteq U$, and $X'$ has the $\mathcal{A}$-extremality property. Fix $n \in \NN$, and mappings $\gamma_i \in \tau_{X'}(Id, \mathcal{A})$, $i \leq n$, as in Lemma \ref{le:main}. Denote the unique extensions of $\gamma_i$ to isometries of $X$ also by $\gamma_i$, and put $V=\tau_X(Id, \mathcal{A})$. Then Lemma \ref{le:main} together with Lemma \ref{le:Count} imply that the orbit of $(\gamma_0, \ldots, \gamma_n)$ under the action of  $V$ is dense in $V^{n+1}$, and that for every open $W \leq V$ the orbit of $(\gamma_0, \ldots, \gamma_n)$ under the action of $W$ is somewhere dense. Thus Point (1) of Lemma \ref{le:fo} implies that the orbit of $(\gamma_0, \ldots, \gamma_n)$ under the action of $V$ is comeager in $V^{n+1}$. Since $n$ is arbitrary, $V$ has ample generics.

The second statement follows from Theorem \ref{th:main:2}. 
\end{proof}

\section{Applications of Theorem \ref{th:main}}

\begin{corollary}
\label{co:Ur}
The isometry group of every Polish ultrametric Urysohn space has ample generics. Moreover, there exists a Polish ultrametric $2$-Urysohn space whose isometry group has ample generics, e.g. $X^2_\mathbbm{Q}$.
\end{corollary}

\begin{proof}
The first statement is obvious in the light of Theorem \ref{th:main} because all $[B]$ are infinite in Polish ultrametric Urysohn spaces. Now fix an FD family $\BB$ in $X^2_\mathbbm{Q}$, and a ball $C$ in $X^2_\mathbbm{Q}$. Observe that $C$ is extremal in every ball $C'$ such that $C \subseteq C'$. Therefore if $C \subseteq B$ for some $B \in \BB$, then $C$ is extremal in $B$, and so is $\BB$-extremal in $X$. Otherwise, let  $r=\dist(C, \BB)$, and let $D$ be the unique ball of radius $r$ such that $C \subseteq D$. As all maximal $r$-polygons in $X^2_\mathbbm{Q}$ have size $2$, $\mathcal{O}^\BB(D)=\{D\}$. Because $C \subseteq D$, $C$ is extremal in $D$, and so it is $\BB$-extremal in $X$ as well.
\end{proof}

\begin{corollary}
\label{co:n-Ur}
If $n>2$, then no isometry group of a Polish ultrametric $n$-Urysohn space $X$ contains an open subgroup with a comeager conjugacy class.
\end{corollary}

\begin{proof}
Let $X=X^n_R$ for a countable $R \subseteq \mathbbm{R}^{>0}$. If $R$ is finite, then $\Iso(X)$ is a finite, discrete group of permutations, so a comeager conjugacy class would be the whole group $\Iso(X)$, which is not possible. We can assume then that $R$ is infinite.

Suppose first that there exists an infinite, strictly increasing sequence $r_0<r_1< \ldots$ in $R$. Let $B_0$ be a ball of radius $r_0$, and let $B_i$, $i>0$, be balls of radius $r_i$ such that $\dist(B_0,B_i)=r_i$. Clearly, each $B_i$, $i>0$, is r-maximal $\{B_0\}$-non-extremal in $X$. Let $\BB$ be an FD family in $X$ strengthening $\{B_0\}$. Then Points (1) and (2) of Lemma \ref{le:max} imply that for almost all $i$ the balls $B_i$ are $\BB$-non-extremal in $X$.

Suppose now that there are no such sequences in $R$. Since $R$ is infinite, there must exists an infinite, strictly decreasing sequence $r_0>r_1> \ldots$ in $R$. Fix  a ball $C_0$ with radius $r_0$, and an FD family $\BB$ in $X$ which strengthens $\{C_0\}$. Using the fact that there are no infinite sequences in $R$, it is easy to construct a sequence $C_0 \supsetneq C_1 \supsetneq \ldots$ of balls such that $C_{k+1}$ is r-maximal non-extremal in $C_k$ for $k \geq 0$. Then following the argument from Case 2 of Theorem \ref{th:main:2} we can prove that there is no comeager conjugacy class in $\Iso(X)$.
\end{proof}

\begin{corollary}
\label{co:co}
There exists a Polish group, which has ample generics, and countable cofinality, e.g. $\Iso(X_\mathbbm{Q})$.
\end{corollary}

\begin{proof}
Fix an increasing sequence $B_0 \subseteq B_1 \subseteq \ldots$ of  balls in $X_\mathbbm{Q}$ such that radius of $B_n$ is $n$, and let $G_n=\tau(Id, \{B_n\})$. Clearly, $G_0<G_1< \ldots$, and $\Iso(X_\mathbbm{Q})= \bigcup G_n$.
\end{proof}

\begin{corollary}
\label{co:de}
There is a locally compact group all of whose diagonal conjugation actions have a dense orbit, e.g., $\Iso(X^2_{\NN})$.
\end{corollary}

\begin{proof}
Put $X=X^2_{\NN}$. It is easy to see that every ball is extremal in $X$, that is, $X$ has the extremality property. As $X$ is countable, Lemma \ref{le:main} applied to $X$ and $\mathcal{A}=\emptyset$ implies that every diagonal conjugation action of $\Iso(X)$ has a dense orbit.
 
In order to see that $\Iso(X)$ is locally compact, fix $x_0 \in X$. Then $\{x_0\}$ is a ball in $X$. Observe that for every $x \in X$ the set
\[ \{g(x): g \in \tau(Id,\{x_0\}) \} \]
is finite. Therefore in every sequence $\{g_n\}$ of elements of $\tau(Id, \{x_0\})$ there exists a convergent subsequence. In other words, $\tau(Id,\{x_0\})$ is compact.
\end{proof}

\section{Projecting Polish ultrametric spaces.}

It has been proved in \cite{MaSo} that for every Polish ultrametric space $X$ there exists a complete pseudo-ultrametric $d_1$ on the set of orbits
\[ X/\Iso(X)=\{ \db{x} : x \in X \} \]
defined by
\[ d_1(\db{x},\db{y})=\dist(\db{x},\db{y}). \]

Moreover, by \cite[Lemma 4.4]{MaSo}, every orbit $\db{x}$ is closed, so $X/\Iso(X)=X/d_1$, where $X/d_1$ is the metric identification of $d_1$. This gives rise to a projection $\pi:X \rightarrow X/d_1$ defined by $x \mapsto \db{x}$. Clearly, this operation can be iterated to construct a sequence of Polish ultrametric spaces and mappings associated with them. 
We show that there also exists a naturally defined limit step, which leads to a countable transfinite sequence of ultrametric spaces whose last element is rigid. Thus, every Polish ultrametric space can be canonically reduced to a rigid one.

Let $(X,d)$ be a Polish ultrametric space. Suppose that $d_\alpha$, for some $\alpha \in Ord$, is a pseudo-ultrametric on $X$. For $x \in X$ we denote by $\db{x}_\alpha$ the equivalence class of $x$ in $X/d_\alpha$. Now we define the following sequence of pseudo-ultrametrics $d_\alpha$, $\alpha \in Ord$, on $X$:
\[ d_0=d,\]
\[ d_{\alpha+1}(x,y)={\rm dist}( \pi(\db{x}_\alpha),\pi(\db{y}_\alpha)), \]
where $\pi$ is the projection associated with $X/d_\alpha$, as described above.

If $\alpha$ is a limit ordinal, we put
\[ d_\alpha(x,y)=\inf \left\{ d_\beta(x,y): \beta<\alpha \right\}. \]

Finally, for $\beta<\alpha$, we define $\pi^\beta_\alpha(\db{x}_\beta)=\db{x}_\alpha$.

\begin{theorem}
\label{th:ri}
Let $X$ be a Polish ultrametric space, and let $d_\alpha$, $\pi^\beta_\alpha$, $\alpha \in Ord$, be defined as above.
\begin{enumerate}
\item $d_\beta(x,y) \geq d_{\alpha}(x,y)$ for every $x,y \in X$ and $\beta<\alpha$;
\item the pseudo-ultrametric $d_\alpha$ induces an ultrametric on $X/d_\alpha$;
\item if $\alpha$ is a limit ordinal, then 
\[ \db{x}_\alpha={\rm cl}_X (\bigcup_{\beta<\alpha} \db{x}_\beta) \]
for every $x \in X$; in particular, $\db{x}_\alpha$ is closed for every $x \in X$;
\item $X/d_\alpha$ is complete, and so it is a Polish ultrametric space;
\item there exists $\alpha_s<\omega_1$ such that $d_{\alpha_s+1}=d_{\alpha_s}$; in other words, $X/d_{\alpha_s}$ is rigid;
\item if $\alpha$ is a limit ordinal, then $X / d_\alpha$ is the direct limit of the system $\left\{ X /d_\beta \right\}_{\beta<\alpha}$ with mappings $\{ \pi^\gamma_\beta \}_{\gamma<\beta<\alpha}$ (that is, $X / d_\alpha$ satisfies a suitable universal property with metric maps as morphisms).
\end{enumerate}
\end{theorem}

\begin{proof}
Point (1) is clear, and Point (2) easily follows from the definition of pseudo-metric.

Points (3) and (4) follow from the fact that if $d_\alpha(\db{x}_\alpha,\db{y}_\alpha)<\epsilon$, then there exists $y' \in \db{y}_\alpha$ with $d(x,y')<\epsilon$. We prove it by a simple induction. For $\alpha=0$, it is obvious. For $\alpha=\gamma+1$, it follows from the definition of $X_\alpha$, and for a limit $\alpha$, there exists $\gamma<\alpha$ such that $d_\gamma(\db{x}_\gamma,\db{y}_\gamma)<\epsilon$, so we can use the inductive assumption.

To prove Point (5), fix a dense subset $\left\{ x_n \right\}_n \subseteq X$. By Point (3), for each $n$ the family $\left\{ \db{x_n}_\alpha \right\}_{\alpha < \omega_1}$ forms an increasing sequence of closed sets, so it stabilizes at some $\alpha_n<\omega_1$. Let $\alpha_s=\sup_n \alpha_n$. The set $\left\{ \db{x_n}_{\alpha_s} \right\}_n$ is a dense subset of $X /d_{\alpha_s}$, and none of its elements is moved by any isometry of $X / d_{\alpha_s}$. Thus Iso$(X/d_{\alpha_s})$ is trivial, and $d_{\alpha_s+1}=d_{\alpha_s}$.

Point (6) directly follows from the definition of the direct limit.
\end{proof}

\end{document}